\documentclass[12pt]{amsart}
\usepackage{amssymb,amsfonts,amsmath,amsthm,latexsym,enumerate,nicefrac,mathrsfs,comment,color}
\usepackage{comment}
\usepackage{bbm}

\theoremstyle{plain}
\newtheorem{theorem}{Theorem}[section]
\newtheorem{lemma}[theorem]{Lemma}
\newtheorem{proposition}[theorem]{Proposition}

\theoremstyle{definition}

\newtheorem{example}[theorem]{Example}

\newcommand{\ba}{\begin{array}{ll}}
\newcommand{\bal}{\begin{array}{ll}}
\newcommand{\ea}{\end{array}}

\numberwithin{equation}{section}

\title[]{Pervasiveness of  $\mathcal{L}^r(E,F)$ in  $\mathcal{L}^r(E,F^{\delta})$}
\author[Q.~Kiervin~Starkey]{Quinn Kiervin Starkey}
\address{Department of Mathematics, Toronto Metropolitan University, 350 Victoria Street, Toronto, Canada M5B 2K3}
\email{quinn.kiervin@torontomu.ca}

\author[F.~Xanthos]{Foivos Xanthos}
\address{Department of Mathematics, Toronto Metropolitan University, 350 Victoria Street, Toronto, Canada M5B 2K3}
\email{foivos@torontomu.ca}

\date{\today}

\begin{document}

\begin{abstract}
Let $E, F$ be Archimedean Riesz spaces, and let $F^{\delta}$ denote an order completion of $F$. In this note, we provide necessary conditions under which the space of regular operators $\mathcal{L}^r(E, F)$ is pervasive in $\mathcal{L}^r(E, F^{\delta})$. Pervasiveness of $\mathcal{L}^r(E, F)$ in $\mathcal{L}^r(E, F^{\delta})$ implies that the Riesz completion of \( \mathcal{L}^r(E, F) \) can be realized as a Riesz subspace of \( \mathcal{L}^r(E, F^{\delta}) \). It also ensures that  the regular part of the space of order continuous operators $\mathcal{L}^{oc}(E, F)$ forms a band of $\mathcal{L}^r(E, F)$. Furthermore,  the positive part $T^+$ of any operator \( T \in \mathcal{L}^r(E, F) \), provided it exists, is given by the Riesz-Kantorovich formula. The results apply in particular to cases where $E = \ell_0^{\infty}$, $E = c$, or $F$ is atomic, and they provide solutions to some problems posed in  \cite{AW:91} and \cite{W:24}.
\end{abstract}



\maketitle

\section{Introduction}
The space of regular operators $\mathcal{L}^r(E, F)$ between Archimedean Riesz spaces $E$ and $F$ has been extensively studied when the range space $F$ is order complete. A wealth of results is available concerning the order structure of $\mathcal{L}^r(E, F)$ in this context. In Theorem~\ref{mainthm}, we show that a couple of key results, namely a version of Ogasawara's theorem and the Riesz-Kantorovich formula continue to hold when the order completeness assumption on $F$ is replaced by the pervasiveness of $\mathcal{L}^r(E, F)$ in $\mathcal{L}^r(E, F^\delta)$. In Proposition~\ref{mainprop1} and Proposition~\ref{mainprop2}, we provide sufficient conditions for the pervasiveness of $\mathcal{L}^r(E, F)$ in $\mathcal{L}^r(E, F^\delta)$. 

The classical Ogasawara theorem states that the space of order continuous operators $\mathcal{L}^{oc}(E, F)$ is a band of $\mathcal{L}^r(E, F)$ whenever $F$ is order complete. It remains an open problem whether this holds when the order completeness assumption on $F$ is relaxed (see \cite{GK:18}, p. 277). In \cite{GK:08, W:24}, the authors explored the case where $E$ is the space $\ell_0^\infty$ of eventually constant real sequences. In \cite{GK:08}, it was shown that $\mathcal{L}^{oc}(\ell_0^\infty, \ell_0^\infty)$ is a band of $\mathcal{L}^r(\ell_0^\infty, \ell_0^\infty)$, and in \cite{W:24}, it was proved that $\mathcal{L}^{oc}(\ell_0^\infty, F)$ is a band of $\mathcal{L}^r(\ell_0^\infty, F)$ if $F$ is an almost Dedekind $\sigma$-complete Riesz space. In Theorem~\ref{maincor2}, we complete the picture by proving that this result holds for any Archimedean Riesz space $F$. We also show in Example~\ref{exmp2} that $\mathcal{L}^{oc}(\ell_0^\infty, F)$ is not necessarily directed. These results resolve  \cite[Problem~4.15]{W:24}. More generally, in Theorem~\ref{mainthm}, we show that pervasiveness of $\mathcal{L}^r(E, F)$ in $\mathcal{L}^r(E, F^\delta)$ implies that $\mathcal{L}^{oc}(E, F) \cap \mathcal{L}^r(E, F)$ is a band of $\mathcal{L}^r(E, F)$. In Example~\ref{exmp1}, we present an example where $\mathcal{L}^{oc}(E, F)$ is not a band of $\mathcal{L}^r(E, F)$, demonstrating that a generalization of Ogasawara's theorem without restricting to the regular part \( \mathcal{L}^{oc}(E, F) \cap \mathcal{L}^r(E, F) \) of  $\mathcal{L}^{oc}(E, F)$ is not possible in general.

In \cite{E:19}, a long-standing problem in the theory of regular operators was resolved by constructing a regular operator whose positive part is not given by the Riesz-Kantorovich formula. In contrast, we demonstrate in Theorem~\ref{mainthm} that pervasiveness of $\mathcal{L}^r(E, F)$ in $\mathcal{L}^r(E, F^\delta)$ guarantees that $\mathcal{L}^r(E,F)$ possesses the Riesz-Kantorovich property. Specifically for each $T \in \mathcal{L}^r(E,F)$ for which $T^+$ exists in $\mathcal{L}^r(E,F)$ we have $T^+(x)=\sup T[0,x]$ for each $x \in E_+$. As an application, Theorem~\ref{maincor1} shows that $\mathcal{L}^r(E, \ell^\infty_0)$ satisfies the Riesz-Kantorovich property, thereby answering a question posed in \cite[p. 262]{AW:91}.

\section{Notations}
Below, we summarize some notations and results that are necessary for the exposition of our findings. The reader is referred to the textbooks \cite{AB:03,AB:06,AL:06,GK:18,SS:74}. for all undefined terms.

\subsubsection*{Pre-Riesz spaces} Let $(X, \geq)$ be a partially ordered vector space, and let $Y$ be  a subspace of $X$. We say that $X$ is a \emph{Riesz space} (or a vector lattice) if for every $x,y \in X$ the set $\{x,y\}$ has a least upper bound (supremum) and a greatest lower bound (infimum). The space $X$ is \emph{Archimedean} if for every $x,y \in X$ with $n x  \leq y$ for all $n \in \mathbb{N}$ one has  $x \leq 0$. We say that $X$ is \emph{directed} (or positively generated) if  for every $x \in X$, there exists $\widetilde{x} \in X$ such that $\widetilde{x} \geq x,0$. The subspace $Y$ is \emph{pervasive} in $X$ if for every $x \in X$ with $x>0$ there exists $y \in Y$ such that $0<y\leq x$. Furthermore, $Y$ is \emph{order dense} in $X$ if $x=\sup\{y \in Y \mid y\leq x\}=\inf\{y \in Y \mid y \geq x\}$ for all $x \in X$. Finally, $Y$ is said to be \emph{majorizing} in $X$ if for every $x \in X$ there is $y \in Y$ such that $y \geq x$.  If $X$ is a Riesz space, then $Y$ is a \emph{Riesz subspace} of $X$ if $x \vee y \in Y$, for all $x,y \in Y$. The \emph{Riesz subspace of $X$ generated by $Y$} is the smallest Riesz subspace that contains $Y$ and is equal to the set of differences of finite suprema of elements of $Y$. Let $A$ be a subset of $X$, we denote by $[A]$ the linear span of $A$, we denote by $A^{u}_X$ (or simply $A^u$) the set of upper bounds of $A$ in $X$, that is $A^u=\{x \in X \mid x \geq a \,\,\text{ for all }  a \in A\}$. The elements $x,y \in X$ are called \emph{disjoint in $X$}, denoted by $x \perp_X y$ (or simply $x \perp  y$), if $\{x+y,-x-y\}^u=\{x-y,-x+y\}^u$. If $X$ is a Riesz space then $x \perp y$ if and only if $|x| \wedge |y|=0$. The \emph{disjoint complement} $Y^d_X$ (or simply $Y^d$) of $Y$ in $X$ is the set $Y^d=\{x \in X \mid x \perp y \,\, \text{ for all } \,\, y \in Y\}$. We say that $Y$ is a \emph{band} of $X$ if $Y^{dd}=Y$ (or equivalently $Y^{dd} \subseteq Y$). The band $Y$ of $X$ is said a \emph{projection band} of $X$ if we have $X=Y \oplus Y^d$.  That is, for every $x \in X$ there exist unique $y \in Y,z \in Y^d$ such that $x=y+z$. The associated projection $P: X \rightarrow Y$ defined by the formula $P(x)=y$ is called the \emph{band projection associated with $Y$}. We say that $X$ is a \emph{pre-Riesz space} if $\{x+z,y+z\}^u \subseteq \{x,y\}^u$ implies $z \in X_+$. By \cite[Proposition 2.2.3]{GK:18}, every pre-Riesz space is directed, and every directed Archimedean  partially ordered vector space  is a pre-Riesz space. Let $X,\widetilde{X}$ be pre-Riesz spaces. We say that $X,\widetilde{X}$ are \emph{order isomorphic} whenever there exists an \emph{order isomorphism} $j: X\rightarrow \widetilde{X}$, i.e.,  a linear bijection $j: X \rightarrow \widetilde{X}$ that is bipositive (i.e., $x \geq 0$ if and only if $j(x) \geq 0$ for all $x \in X$). If moreover $X,\widetilde{X}$ are Riesz spaces, then we say that $j$ is a \emph{Riesz isomorphism}. We say that $X^{\rho}$ is a Riesz completion of $X$ whenever $X^{\rho}$ is a Riesz space and there exists a bipositive linear map $i:X \rightarrow X^{\rho}$ such that $i(X)$ is order dense in $X^{\rho}$ and the Riesz subspace of $X^{\rho}$ generated by $i(X)$ is equal to $X^{\rho}$. Riesz completions are unique up to Riesz isomorphisms (see \cite[Proposition~2.4.4]{GK:18}).

\subsubsection*{Riesz spaces}

For the following $X$ will denote an Archimedean Riesz space. We say that $X$ is \emph{order complete} if every nonempty subset of $X$ that is bounded from above has a supremum. The order complete Riesz space $X^{\delta}$ is an \emph{order completion}  of $X$ if there exists a linear bipositive map $i: X \rightarrow X^{\delta}$ such that $i(X)$ is both majorizing and pervasive in $X^{\delta}$, order completions are unique up to Riesz isomorphisms. We note that $X^{\delta}$ is also Archimedean (see, for example,~\cite[Proposition~2.4.12]{GK:18}). We say that a net $(x_\alpha)_{\alpha \in A}$ in $X$ \emph{converges  uniformly} or relatively uniformly to $x \in X$ if there exists $v \in X_+$ such that for all $\epsilon>0$ there exists $\alpha_0 \in A$ such that $|x_\alpha-x| \leq \epsilon v$ for all $\alpha \geq \alpha_0$. In this case, we write $x_{\alpha} \xrightarrow{{u}} x$. We say that a sequence $(x_{n})$ in $X$ is \emph{uniformly Cauchy} if there exists $v \in X_+$ such that for all $\epsilon>0$ there exists $n_0$ such that $|x_{n}-x_{m}|\leq \epsilon v$ for all $n,m \geq n_0$.  We say that $X$ is \emph{uniformly complete} if every uniformly Cauchy sequence in $X$ converges uniformly in $X$. For any set $K \subseteq X$ we denote the sets of uniformly limits of all sequences in $K$ by $\overline{K}^{u}=\{x \in X \mid x_{n}\xrightarrow{{u}} x \,\, \text{ for some } (x_{n}) \subseteq K\}$ and call it the \emph{uniformly closure} of $K$ in $X$. A net $(x_{\alpha})_{\alpha \in A}$ in $X$ is said to \emph{converge in order} to $x \in X$, written $x_{\alpha} \xrightarrow{{o}} x$ if there exists another net $(b_{\gamma})_{\gamma \in \Lambda}$ in $X$ such that $b_{\gamma} \downarrow 0$ and for any $\gamma \in \Lambda$ there exists $\alpha_0 \in A$ such that $|x_{\alpha}-x|\leq b_{\gamma}$ for all $\alpha \geq \alpha_0$. Order convergence forms a Hausdorff convergence structure. The order convergence structure is linear and the lattice operations are order continuous, moreover uniform convergence is the Mackey modification of the order convergence structure  (see \cite{BTW23} for more details). We will also make use of a variant of order convergence. For a net  $(x_{\alpha})_{\alpha \in A}$ in $X$ and $x \in X$ we write $x_{\alpha} \xrightarrow{{o_1}} x$ if there exists a net $(y_{\alpha})_{\alpha \in A}$ in $X$ such that $y_{\alpha} \downarrow 0$ and $|x_{\alpha}-x| \leq y_{\alpha}$ for all $\alpha \in A$. We caution the reader that the notation we use here for $o$-convergence and  $o_1$-convergence is not in agreement with \cite{GK:18}. If $X$ is order complete and $(x_{\alpha})$ is an order bounded net in $X$ then $x_{\alpha} \xrightarrow{{o_1}} 0$ if and only if  $x_{\alpha} \xrightarrow{{o}} 0$ (see, for example,~\cite[Proposition~3.7.3]{GK:18}), however in general $o$-convergence does not agree with $o_1$-convergence. We say that $\{e_{i} \,\, | \,\, i \in I\}$ is a \emph{complete disjoint system} in $X$ if $e_{i}>0$ for each $i \in I$, $e_i \wedge e_j=0$ for each $i \neq j$, and if $x \wedge e_i=0$ for all $i \in I$, then $x=0$. An element $e$ in $X_+$ is said to be an atom if $0 \leq x \leq e$ implies that $x$ is a scalar multiple of $e$, in this case the subspace $[e]$ is a projection band of $X$ and we denote by $P_e$ the band projection associated with $[e]$.  We say that $X$ is \emph{atomic} if $X$ has a complete disjoint system consisting of atoms. Let $X$ be an atomic space and  $\{e_{i} \,\, | \,\, i \in I\}$  be a complete disjoint system of atoms then $x=\bigvee_{i \in I} P_{e_i}(x)=\bigvee_{i \in I} \lambda_{e_i}(x) e_i$, for every $x \in X_+$, where $\lambda_{e_i}: X \rightarrow \mathbb{R}$ is the \emph{coordinate functional of $e_i$} (see, for example,~\cite[Theorem~1.75, Theorem~1.77]{AB:03}). We write $\mathrm{Fin}(I)$ to denote the collection of all finite subsets of the index set $I$ and we view $\mathrm{Fin}(I)$ as a directed set ordered by  inclusion.

\subsubsection*{Space of  operators}

Let $E,F$ be Archimedean Riesz spaces, we denote by $\mathcal{L}^r(E,F)$  the space of regular operators between $E$ and $F$. Note that $\mathcal{L}^r(E,F)$ equipped with the pointwise partial order is a directed  Archimedean  partially ordered vector space (see, for example \cite[Proposition 1.2.4]{GK:18}), and thus $\mathcal{L}^r(E,F)$ is a pre-Riesz space. The space of order-bounded operators between $E$ and $F$ is denoted by $\mathcal{L}^b(E,F)$. By \cite[Theorem~5.1]{TT:20} we have that a linear operator $T:E \rightarrow F$ is order bounded if and only if $T$ preserves uniform convergence, that is $x_{\alpha} \xrightarrow{{u}} 0$ implies $T(x_{\alpha}) \xrightarrow{{u}} 0$ for every net $(x_{\alpha})$ in $E$. We denote the space  $\mathcal{L}^b(E,\mathbb{R})$ by $E^\sim$. If $F$ is an order-complete space then $\mathcal{L}^r(E,F)=\mathcal{L}^b(E,F)$ is an order-complete Riesz space and the lattice operations are given by the Riesz-Kantorovich formula (see, for example \cite[Theorem 1.67]{AB:03}).  We say that a linear operator $T:E \rightarrow F$ is order continuous  if for any net $(x_{\alpha})_{\alpha \in A}$ in $E$ such that $x_{\alpha} \xrightarrow{{o}} 0$ we have $T(x_{\alpha}) \xrightarrow{{o}} 0$. We denote the  space of order-continuous operators between $E$ and $F$ by $\mathcal{L}^{oc}(E,F)$. In \cite{AB:03,AB:06,AL:06,GK:18}\footnote{The definition of the space of order-continuous operators in \cite{W:24} is consistent with our paper.} the authors used the term ``order continuity" and the notation $\mathcal{L}^{oc}(E,F)$ (or $\mathcal{L}_n(E,F)$) to denote the space of linear operators that are $o_1$-continuous (i.e., preserve $o_1$-convergence). In general, $o_1$-continuity is not equivalent to order continuity. However, both continuities imply that the operator $T$ is order bounded (see \cite[Theorem~2.1]{AS:05},  \cite[Lemma~1.54]{AL:06}).  If $T$ is positive, then $o_1$-continuity is equivalent to order continuity, and both continuities are equivalent to the implication $x_{\alpha} \downarrow 0 \Rightarrow T(x_{\alpha}) \downarrow 0$, for any net $(x_{\alpha})$ in $E$ (see  \cite[Theorem 4.4]{{HK:21}}). If $F$ is order complete, then $\mathcal{L}^{oc}(E,F)$ coincides with the space of linear operators that are $o_1$-continuous  (see  \cite[Theorem 9.6]{{HK:21}}). In light of the above, the classical  Ogasawara theorem, stated for $o_1$-continuous operators (see, for example \cite[Theorem 1.73]{AB:03}) implies that $\mathcal{L}^{oc}(E,F)$ is a projection band of $\mathcal{L}^r(E,F)$ whenever $F$ is order complete.

\section{The main results}
In what follows, unless otherwise stated, $E$ and $F$ denote Archimedean Riesz spaces, $F^{\delta}$ an order completion of $F$, and $i: F \to F^{\delta}$ the corresponding bipositive map. We will say that \textbf{$\mathcal{L}^r(E,F)$ is pervasive in $\mathcal{L}^r(E,F^{\delta})$} if $\mathcal{L}^r(E,i(F))$ is pervasive in $\mathcal{L}^r(E,F^{\delta})$, we also say that $\mathcal{L}^r(E,F)$ has the \textbf{Riesz-Kantorovich property} whenever for each $T \in \mathcal{L}^r(E,F)$ such that $T^+$ exists in $\mathcal{L}^r(E,F)$ we have  $T^+(x)=\sup\{T(x) \,\, | \,\, y \in [0,x]\}$ for each $x \in E_+$.

We begin with the following lemma, which enables us to work with the space $\mathcal{L}^r(E,i(F))$. The proof is straightforward and relies on standard results: (i) order isomorphisms between pre-Riesz spaces preserve arbitrary suprema, and (ii) they are disjointness-preserving operators (see, for example,~\cite[Proposition 5.2.4, Theorem 2.3.19, Theorem 5.1.12]{GK:18}). Additionally, Riesz isomorphisms are order continuous (see, for example,~\cite[Theorem 2.21]{AB:03}).

\begin{lemma}\label{lemma00000}
Let $\widetilde{F}$ be a Riesz space that is order isomorphic to $F$. Then the spaces $\mathcal{L}^r(E,F)$ and  $\mathcal{L}^r(E,\widetilde{F})$ are order isomorphic, moreover we have
\begin{enumerate}
\item[(i)] $\mathcal{L}^r(E,F)$ has the  Riesz-Kantorovich property if and only if $\mathcal{L}^r(E,\widetilde{F})$ has the  Riesz-Kantorovich property.
\item[(ii)]  $\mathcal{L}^{oc}(E,F)\cap \mathcal{L}^r(E,F)$ is a band of $\mathcal{L}^r(E,F)$ if and only if $\mathcal{L}^{oc}(E,\widetilde{F})\cap \mathcal{L}^r(E,\widetilde{F})$ is a band of $\mathcal{L}^r(E,\widetilde{F})$.
\end{enumerate}
\end{lemma}

\begin{proof}
Let $l:F \rightarrow \widetilde{F}$ be an order isomorphism. We define the operator $j:\mathcal{L}^r(E,F) \rightarrow \mathcal{L}^r(E,\widetilde{F})$ with $j(T)(x)=l(T(x))$ for all $x \in E$ and $T \in \mathcal{L}^r(E,F)$. Then it is easy to see that $j$ is linear, surjective  and bipositive, thus  $j$ is an order isomorphism and the spaces $\mathcal{L}^r(E,F)$ and  $\mathcal{L}^r(E,\widetilde{F})$ are order isomorphic. 

(i) Suppose that $\mathcal{L}^r(E,\widetilde{F})$ has the  Riesz-Kantorovich property, and let $T \in \mathcal{L}^r(E,F)$ such that $T^+$ exists. Then since $j$ preserves arbitrary suprema we have that $j(T)^+$ exists and is equal to $j(T^+)$. In particular we get $j(T)^+(x)=j(T^+)(x)=l(T^+(x))$ for each $x \in E$. Thus by using the Riesz-Kantorovich formula for $j(T)^+$ and the fact that both $l$ and $l^{-1}$ preserve arbitrary suprema we get 
\begin{align*}T^+(x)&=l^{-1}(j(T)^+(x))=l^{-1}\big(\sup\{j(T)(y) \mid y \in [0,x]\}\big) \\
&=l^{-1}(\sup\{l(T(y)) \mid y \in [0,x]\})=\sup\{T(y) \mid y \in [0,x]\},
\end{align*}
for all $x \in E_+$. Thus $\mathcal{L}^r(E,F)$ has the  Riesz-Kantorovich property.

(ii) Let $\mathcal{A}$ be a nonempty set of $\mathcal{L}^r(E,\widetilde{F})$, we claim that 
\begin{equation}\label{eqband}j^{-1}(\mathcal{A}^d)=j^{-1}(\mathcal{A})^d.\end{equation}

Indeed, let $T \in j^{-1}(\mathcal{A}^d)$, then $j(T) \in \mathcal{A}^d$. Now, let $S \in j^{-1}(\mathcal{A})$, then $j(S) \in \mathcal{A}$ and thus $j(S) \perp j(T)$. Since $j^{-1}$ is a disjointness preserving operator we get $S  \perp  T$ and thus $T \in j^{-1}(\mathcal{A})^d$. In particular, we have  $ j^{-1}(\mathcal{A}^d) \subseteq  j^{-1}(\mathcal{A})^d$. Conversely, suppose that $T \in  j^{-1}(\mathcal{A})^d$ then $T \perp j^{-1}(S)$ for all $S \in \mathcal{A}$. Since $j$ is  a disjointness preserving operator we get $j(T) \perp S$ for all $S \in \mathcal{A}$. In particular, $j(T) \in \mathcal{A}^d \Rightarrow T \in j^{-1}(\mathcal{A}^d)$, thus $j^{-1}(\mathcal{A})^d \subseteq j^{-1}(\mathcal{A}^d)$ and ~(\ref{eqband}) is established. Now, suppose that $\mathcal{B}=\mathcal{L}^{oc}(E,\widetilde{F})\cap \mathcal{L}^r(E,\widetilde{F})$ is a band of $\mathcal{L}^r(E,\widetilde{F})$. Then we have $\mathcal{B}^{dd}=\mathcal{B}$. Thus by applying repeatedly ~(\ref{eqband}) we get

$$j^{-1}(\mathcal{B})^{dd}=j^{-1}(\mathcal{B}).$$

Therefore  $j^{-1}(\mathcal{B})$ is a band of $\mathcal{L}^r(E,F)$. Finally, we note $j^{-1}(\mathcal{B})=\mathcal{L}^{oc}(E,F)\cap \mathcal{L}^r(E,F)$. Indeed, since $l$ and $l^{-1}$ are order-continuous operators we have $T \in \mathcal{L}^{oc}(E,F)\cap \mathcal{L}^r(E,F)$ if and only if $jT \in \mathcal{B}$ for any $T \in \mathcal{L}^r(E,F)$
\end{proof}

The following theorem is the main result of our paper. It demonstrates that the pervasiveness of $\mathcal{L}^r(E,F)$ in $\mathcal{L}^r(E,F^{\delta})$ enables the derivation of results typically obtained under the assumption that $F$ is order complete.

\begin{theorem}\label{mainthm}
If $\mathcal{L}^r(E,F)$ is pervasive in $\mathcal{L}^r(E,F^{\delta})$, then $\mathcal{L}^r(E,F)^\rho$ can be identified with the Riesz subspace of $\mathcal{L}^r(E,F^{\delta})$ generated by $\mathcal{L}^r(E,i(F))$. Moreover, $\mathcal{L}^r(E,F)$ has  the Riesz-Kantorovich property, and $\mathcal{L}^{oc}(E,F)\cap \mathcal{L}^r(E,F)$ is a band of $\mathcal{L}^r(E,F)$.
\end{theorem}

\begin{proof}
In light of Lemma~\ref{lemma00000}, we may assume, without loss of generality, that $i(F)=F$. Let $\mathcal{S}$ be the Riesz subspace of $\mathcal{L}^r(E,F^{\delta})$ generated by $\mathcal{L}^r(E,F)$. We will show next that $\mathcal{L}^r(E,F)$ is majorizing in $\mathcal{S}$. Indeed, let $T \in \mathcal{S}$, then $$T=\bigvee_{i=1}^k V_i -\bigvee_{j=1}^m W_j,$$ where $V_i,W_j \in \mathcal{L}^r(E,F)$ for all $i=1,\cdots, k$ and $j=1,\cdots,m$. Since $\mathcal{L}^r(E,F)$ is directed we may pick $\widetilde{V}_i, \widetilde{W}_j \in \mathcal{L}^r(E,F)$ such that $\widetilde{V}_i \geq V_i, 0$ for each $i=1,\cdots,k$ and $\widetilde{W}_j \geq -W_j, 0$ for each $j=1,\cdots,m$. Then we note that

$$\widetilde{V}_1+\cdots+\widetilde{V}_k\geq \widetilde{V}_i \geq V_i \text{ and } W_j \geq -\widetilde{W}_j  \geq -(\widetilde{W}_1+\cdots +\widetilde{W}_m) ,  $$

for all $i=1,\cdots,m$ and $j=1,\cdots,k$. Thus we get 

$$\sum_{i=1}^k \widetilde{V}_i \geq \bigvee_{i=1}^k V_i \text{ and } \bigvee_{j=1}^m W_j \geq -\sum_{j=1}^m \widetilde{W}_j. $$

In particular, it follows that 
$$\sum_{i=1}^k \widetilde{V}_i+\sum_{j=1}^m \widetilde{W}_j \geq \bigvee_{i=1}^k V_i-\bigvee_{j=1}^m W_j =T.  $$

Therefore $\mathcal{L}^r(E,F)$ is both majorizing and pervasive in the Archimedean Riesz space $\mathcal{S}$. By applying \cite[Proposition~2.8.5]{GK:18} we get that $\mathcal{L}^r(E,F)$ is order dense in $\mathcal{S}$. Thus the Riesz completion $\mathcal{L}^r(E,F)^\rho$ of $\mathcal{L}^r(E,F)$ can be identified with $\mathcal{S}$.  

Let $T \in \mathcal{L}^r(E,F)$ such that $T^+$ exists. Then by applying~\cite[Proposition~1.6.2]{GK:18} and the Riesz-Kantorovich formula calculated in $\mathcal{L}^r(E,F^{\delta})$ we get  $T^+(x)=\sup\{T(y) \,\, | \,\, y\in [0,x]\} $ for all $x \in E_+$.  Thus $\mathcal{L}^r(E,F)$ has the Riesz-Kantorovich property.

Next we will show that $\mathcal{L}^{oc}(E,F)\cap \mathcal{L}^r(E,F)$ is a band of $\mathcal{L}^r(E,F)$.  By the classical  Ogasawara theorem we have that $\mathcal{L}^{oc}(E,F^{\delta}) $ is a band of $\mathcal{L}^r(E,F^\delta)$. We claim that $\mathcal{B}=\mathcal{S} \cap \mathcal{L}^{oc}(E,F^{\delta})$ is a band of $\mathcal{S}$. We note first that  since $\mathcal{S}$ is a Riesz subspace of $\mathcal{L}^r(E,F^{\delta})$ we have $\mathcal{L}^{oc}(E,F^{\delta})^{d}_{\mathcal{X}}\cap \mathcal{S} \subseteq \mathcal{B}^{d}_{\mathcal{S}}$, where  $\mathcal{X}=\mathcal{L}^r(E,F^\delta)$. Let $S \in \mathcal{S}$ such that $S \in (\mathcal{B}^{d}_{\mathcal{S}})_{\mathcal{S}}^d$ and $S \not\in \mathcal{B}$. Then $S \not\in \mathcal{L}^{oc}(E,F^{\delta})=(\mathcal{L}^{oc}(E,F^{\delta})^{d}_{\mathcal{X}})_{\mathcal{X}}^d$. Thus there exists $T\in  \mathcal{L}^{oc}(E,F^{\delta})^d_{\mathcal{X}}$ such that $T \not\perp_{\mathcal{X}} S$. Then $|T|\wedge |S|>0$ and by pervasiveness of $\mathcal{L}^r(E,F)$ in $\mathcal{L}^r(E,F^{\delta})$ we can find $S_0 \in \mathcal{L}^r(E,F) \subseteq \mathcal{S}$ such that

$$0<S_0 \leq |T| \wedge |S| \leq |T|,|S|.$$

Since $T \in \mathcal{L}^{oc}(E,F^{\delta})^d_{\mathcal{X}}$ and $S_0=|S_0| \leq |T|$ we have $S_0 \in \mathcal{L}^{oc}(E,F^{\delta})^d_{\mathcal{X}}$ and thus $S_0 \in \mathcal{B}^{d}_{\mathcal{S}}$, which is a contradiction as $|S| \wedge |S_0|=S_0>0$ and $S \in (\mathcal{B}^{d}_{\mathcal{S}})_{\mathcal{S}}^d$. Thus $\mathcal{B}$ is a band of $\mathcal{S}$. Now by applying  \cite[Corollary~4.2.7]{GK:18} we get that $\mathcal{B} \cap \mathcal{L}^r(E,F) $ is a band of $\mathcal{L}^r(E,F)$.  Finally we note that $\mathcal{B} \cap \mathcal{L}^r(E,F)= \mathcal{L}^{oc}(E,F) \cap \mathcal{L}^r(E,F)$. Indeed, let $T \in \mathcal{L}^r(E,F)$ and $(x_a) \subseteq E$  such that $x_{\alpha} \xrightarrow{{o}} 0$, then by \cite[Corollary~2.9]{GTX:17} we have  $T(x_{\alpha}) \xrightarrow{{o}} 0$ in $F^{\delta}$ if and only if  $T(x_{\alpha}) \xrightarrow{{o}} 0$ in $F$. 
\end{proof}

The above result prompt us to identify conditions on $E$ and $F$ under which $\mathcal{L}^r(E,F)$ is pervasive in $\mathcal{L}^r(E,F^{\delta})$. Let $f \in E^{\sim}, v \in F$, then we denote by $f \otimes v$ the rank one operator given by the formula $(f \otimes v)(x)=f(x) v$ for all $x \in E$, and we denote by $\mathcal{RO}(E,F)=\{f \otimes v \,\, | \,\, f\in E^\sim,v\in F\}$ the set of all order bounded rank one operators.

\begin{lemma}\label{lemmanew0}
If $\mathcal{RO}(E,F^{\delta})$ is pervasive in $\mathcal{L}^r(E,F^{\delta})$, then $\mathcal{L}^r(E,F)$ is pervasive in $\mathcal{L}^r(E,F^{\delta})$.
\end{lemma}

\begin{proof}
Let  $T \in \mathcal{L}^r(E,F^{\delta})$ with $T>0$. Then we may find $f \in E^\sim\setminus\{0\}, v \in F^{\delta}\setminus\{0\}$ such that $0<S=f \otimes v\leq T$. Let $x_0 \in E_+$ such that $S(x_0)=f(x_0) v \in F_+^{\delta} \setminus\{0\}$. Define $\widetilde{f} \in E^{\sim}$ by $\widetilde{f}(x)=f(x)/f(x_0)$ for all $x \in E$ and note that $$S=f \otimes v=\widetilde{f} \otimes S(x_0).$$ Since $S>0$ and $S(x_0)\in F_+^{\delta}\setminus\{0\}$, we get $\widetilde{f}>0$.  Now  we choose $\widetilde{v} \in i(F)$ such that $0< \widetilde{v} \leq S(x_0)$ and put $\widetilde{S}=\widetilde{f} \otimes \tilde{v},$    then clearly $\widetilde{S} \in \mathcal{L}^r(E,i(F))$ and $0 <\widetilde{S}\leq S \leq T.$  It now follows from the above that  $\mathcal{L}^r(E,F)$ is pervasive in $\mathcal{L}^r(E,F^{\delta})$.
\end{proof}

In the following Proposition, we provide a sufficient condition on the range space \( F \) under which \( \mathcal{L}^r(E,F) \) is pervasive in \( \mathcal{L}^r(E,F^{\delta}) \). To this end, we recall that the coordinate functional \( \lambda_{e_i} \), corresponding to an atom \( e_i \) as defined via the band projection associated with \( [e_i] \), is both positive and order-continuous.

\begin{proposition}\label{mainprop1}
If $F$ is atomic, then $\mathcal{L}^r(E,F)$ is pervasive in $\mathcal{L}^r(E,F^{\delta})$.
\end{proposition}

\begin{proof}
Without loss of generality, we assume that $i(F)=F$. Let $\{e_{i} \,\, | \,\, i \in I\}$ be a complete disjoint system of atoms in $F$, then it is easy to see that  $\{e_{i} \,\, | \,\, i \in I\}$ is also a complete disjoint system of atoms in $F^{\delta}$.  Indeed, let $y_0 \in F^{\delta}$ such that $0 \leq y_0 \leq e_i$. Since $F$ is order dense in $F^{\delta}$ we get $y_0=\sup\{y \in F \mid 0\leq y\leq y_0\}$. Now for any $y \in F$ such that $0\leq y\leq y_0 \leq e_i$ we have  $y=\lambda_{y}e_i$ for some $\lambda_{y} \in [0,1]$. Put $k$ to be the supremum of all such $\lambda_y$. Then, since $F^{\delta}$ is also Archimedean we get $y_0=\sup\{y \in F \mid 0\leq y\leq y_0\}=ke_i$. Hence $e_i$ is an atom of $F^{\delta}$.  Now let $y_0 \in F^{\delta}$ such that $|y_0| \wedge e_i=0$ for all $i$ and suppose that $y_0 \neq 0$. Since $F$ is pervasive in $F^{\delta}$ we may find $y \in F$ such that $0<y\leq |y_0|$, then $y \wedge e_i=0$ for all $i$, thus $y=0$, which is a contradiction. Thus  $\{e_{i} \,\, | \,\, i \in I\}$ is  a complete disjoint system of atoms in $F^{\delta}$. Let $\lambda_{e_i}$ be the coordinate functional of $e_i$ and  $T \in \mathcal{L}^r(E,F^{\delta})$ with $T>0$. Then there exists $x_0 \in E_+$ with $T(x_0)>0$, thus we can find $j \in I$ with $\lambda_{e_{j}}(T(x_0))>0$. Put $\widetilde{T}=(\lambda_{e_{j}} \circ T) \otimes e_{j}$, then we clearly have that $\widetilde{T} \in \mathcal{RO}(E,F^{\delta})$ and $\widetilde{T}>0$. Moreover, for any $x \in E_+$ it follows that 

$$\widetilde{T}(x)=\lambda_{e_{j}}(T(x))e_{j} \leq \bigvee_{i \in I} \lambda_{e_i}(T(x)) e_i=T(x),$$

Therefore  $\mathcal{RO}(E,F^{\delta})$ is pervasive in $\mathcal{L}^r(E,F^{\delta})$ and the result follow by Lemma~\ref{lemmanew0}.
\end{proof}

We now turn our focus to the domain space $E$. The results in \cite{W:24} motivated us to study the case where $E$ is atomic. As demonstrated in the following result, the band projection associated with $\mathcal{L}^{oc}(E,F^{\delta})$ takes a specific form in this setting, allowing for a straightforward verification of whether an order-bounded operator $T:E \rightarrow F$ is order continuous.

\begin{lemma}\label{lemmanew}
Let $F$ be an order complete Archimedean Riesz space. Let $\{e_{i} \mid i \in I\}$ be a complete disjoint system of atoms in $E$, and let $\lambda_{e_i}$ denote the coordinate functional of $e_i$. Let $T \in \mathcal{L}^r(E,F)$, $P$ be the band projection associated with $\mathcal{L}^{oc}(E,F)$, and define $T_{\alpha} = \sum_{i \in \alpha} \lambda_{e_i} \otimes T(e_i)$ for all $\alpha \in \mathrm{Fin}(I)$. Then the following hold:
\begin{enumerate}
    \item[(i)] $T^\pm(e_i) = T(e_i)^\pm$ for each $i \in I$.
    \item[(ii)] $T_\alpha \xrightarrow{o} P(T)$, and in particular, $T_\alpha(x) \xrightarrow{o} P(T)(x)$ for all $x \in E$. Furthermore, if $T \geq 0$, the convergence $\xrightarrow{o}$ can be replaced by $\uparrow$ in the preceding limits.
\end{enumerate}
\end{lemma}
\nopagebreak
\begin{proof}
(i) Since $e_i$ is an atom, by applying the Riesz-Kantorovich formula, we get
$$ T^+(e_i) = \sup\{T(x) \mid 0 \leq x \leq e_i\} = \sup\{\lambda T(e_i) \mid 0 \leq \lambda \leq 1\} = \sup\{T(e_i), 0\} = T(e_i)^+. $$

We also have $T^-(e_i) = (-T)^+(e_i) = (-T(e_i))^+ = T(e_i)^-$ for each $i \in I$.

(ii) We assume first that $T \geq 0$. We will show that $T_\alpha \uparrow P(T)$. By \cite[Theorem~1.43]{AL:06}, we have
$$ P(T) = \sup\{S \in \mathcal{L}^{oc}(E,F) \mid 0 \leq S \leq T\}. $$

We note that the net $(T_\alpha)_{\alpha \in \mathrm{Fin}(I)}$ is increasing and $T_\alpha = \sum_{i \in \alpha} \lambda_{e_i} \otimes T(e_i) \in \mathcal{L}^{oc}(E,F)$ for each $\alpha$, as $T_{\alpha}$ is a finite sum of the order continuous operators $\lambda_{e_i} \otimes T(e_i)$. For all $x \in E_+$ and $\alpha \in \mathrm{Fin}(I)$, we have
$$ 0 \leq T_\alpha(x) = \sum_{i \in \alpha} \lambda_{e_i}(x) \otimes T(e_i) = T\left(\sum_{i \in \alpha} \lambda_{e_i}(x) e_i\right) \leq T\left(\bigvee_{i \in I} \lambda_{e_i}(x) e_i\right) = T(x). $$

Therefore, $0 \leq T_\alpha \leq P(T)$ for each $\alpha \in \mathrm{Fin}(I)$. The order completeness of $\mathcal{L}^r(E,F)$ implies that there exists $W \in \mathcal{L}^r(E,F)$ such that $T_\alpha \uparrow W \leq P(T)$. Thus, we have $W(x) = \sup_{\alpha \in \mathrm{Fin}(I)} T_\alpha(x)$ for each $x \in E_+$ (see~\cite[Theorem~1.67(b)]{AB:03}). Let $S \in \mathcal{L}^{oc}(E,F)$ such that $0 \leq S \leq T$. Then for any $\alpha \in \mathrm{Fin}(I)$ and $x \in E_+$, we have 
$$ S\left(\sum_{i \in \alpha} \lambda_{e_i}(x) e_i\right) \leq T_\alpha(x) \leq W(x). $$
Since $x= \bigvee_{i \in I} \lambda_{e_i}(x) e_i=\sup\{ \sum_{i \in \alpha} \lambda_{e_i}(x) e_i \mid \alpha \in \mathrm{Fin}(I)\}$ we have $\sum_{i \in \alpha} \lambda_{e_i}(x) e_i \uparrow x $. By applying the order-continuity of $S$ we get
$$ S(x) =\sup_{\alpha \in \mathrm{Fin}(I)} S\left(\sum_{i \in \alpha} \lambda_{e_i}(x) e_i\right)\leq W(x). $$

Therefore $P(T) \leq W$, and thus $W = P(T)$.

Now let $T \in \mathcal{L}^r(E,F)$ and note that $$T_\alpha=\sum_{i \in \alpha} \lambda_{e_i} \otimes T(e_i) =\sum_{i \in \alpha} \lambda_{e_i}\otimes (T(e_i))^+-\sum_{i \in \alpha} \lambda_{e_i}\otimes (T(e_i))^-.$$ Thus by applying (i) we get that 
$$T_\alpha=\sum_{i \in \alpha} \lambda_{e_i} \otimes T^+(e_i)-\sum_{i \in \alpha} \lambda_{e_i} \otimes T^-(e_i)=T^+_\alpha-T^-_\alpha
\overset{o}{\longrightarrow}P(T^+)-P(T^-)=P(T).$$

Finally, let $x \in E$. By the Riesz-Kantorovich formula, we have
$$ |T_\alpha(x) - P(T)(x)| \leq |T_\alpha - P(T)|(|x|).$$ Since $T_\alpha \xrightarrow{o} P(T)$, by invoking again \cite[Theorem~1.67(b)]{AB:03}, we get $T_\alpha(x) \xrightarrow{o} P(T)(x)$.
\end{proof}

\begin{proposition}\label{mainprop2}
If  $E$ is atomic and the uniform closure of the span of its atoms is equal to $E$ or has codimension $1$, then $\mathcal{L}^r(E,F)$ is pervasive in $\mathcal{L}^r(E,F^{\delta})$.
\end{proposition}

\begin{proof}
We will show next that  $\mathcal{RO}(E,F^{\delta})$ is pervasive in $\mathcal{L}^r(E,F^{\delta})$. The result will then follow by Lemma~\ref{lemmanew0}.  Let $\{e_{i} \,\, | \,\, i \in I\}$ be a complete disjoint system of atoms in $E$, $\lambda_{e_i}$ be the coordinate functional of $e_i$ and $P$ the band projection associated with $\mathcal{L}^{oc}(E,F^{\delta})$. 

Let $T \in \mathcal{L}^r(E,F^\delta)$ with $T > 0$. We begin by considering  the case where $T(e_j)>0$ for some $j \in I$. Define $\widetilde{T}=\lambda_{e_{j}} \otimes  T(e_{j}) \in  \mathcal{RO}(E,F^{\delta})$. By Lemma~\ref{lemmanew}(ii) we have $$0<\widetilde{T}=\lambda_{e_{j}} \otimes  T(e_{j}) \leq \sup\{\sum_{i \in \alpha} \lambda_{e_{i}} \otimes  T(e_{i}) \mid \alpha \in Fin(I)\}=P(T) \leq T$$. 

Next, suppose that $T(e_i)=0$ for all $i$.  By \cite[Theorem~5.1]{TT:20}, $T$  preserves uniform convergence, thus $T$ is zero on the uniform closure   $\overline{[\{e_i\}]}^{u}$ of the span of $\{e_i \,\, | \,\, i \in I\}$. If 
$E=\overline{[\{e_i\}]}^{u}$ we get that $T=0$,  which leads to a contradiction. If $E=\overline{[\{e_i\}]}^{u} \oplus [x_0]$, then $T$ is a non-zero rank-one operator. In particular there exists $y \in F^{\delta}\setminus\{0\}$ and $f:E \rightarrow \mathbb{R}$ linear such that $T(x)=f(x)y$ for all $x \in E$.  We claim that $f$ is order bounded. Indeed, suppose not, then there exists $x \in E_+$ and $(x_n) \subseteq E$ such that $|x_n|\leq x$ and $|f(x_n)| \geq n$ for all $n$. Since $T>0$ we can find $w \in F^{\delta}_+$ such that $|T(x_n)| \leq w$ for all $n \in \mathbb{N}$. Thus we have that $$n|y| \leq |T(x_n)|=|f(x_n)||y|\leq w,$$
for all $n \in \mathbb{N}$. By the Archimedean property of $F^{\delta}$ we get $y=0$, which is a contradiction,  thus $f \in E^{\sim}$ and $T  \in  \mathcal{RO}(E,F^{\delta})$. In both cases we have found $\widetilde{T} \in \mathcal{RO}(E,F^{\delta})$ such that $0 <\widetilde{T}  \leq T$, establishing the pervasiveness of $\mathcal{RO}(E,F^{\delta})$ in $\mathcal{L}^r(E,F^{\delta})$.
\end{proof}

A well-known case where the space $\mathcal{L}^r(E,F)$ is a Riesz space without $F$ being order complete is when $E,F$ are Banach lattices and $E$ is atomic with an order continuous norm (see \cite[Theorem 2.4]{W:19}). Below we give a variant of this result  under slightly weaker assumptions.

\begin{theorem}\label{mainthm2}
Let $E$ be an atomic space such that the uniform closure of the span of atoms is equal to $E$. If $F$ is uniformly complete, then $\mathcal{L}^r(E,F)$ is a Riesz space.
\end{theorem}

\begin{proof}
By Lemma~\ref{lemma00000} the spaces $\mathcal{L}^r(E,F)$ and $\mathcal{L}^r(E,i(F))$ are order isomorphic. Therefore, without loss of generality, we may assume that $F=i(F)$. We will show next that $\mathcal{L}^r(E,F)$ is a Riesz subspace of $\mathcal{L}^r(E,F^{\delta})$. Let $\{e_{i} \,\, | \,\, i \in I\}$ be a complete disjoint system of atoms in $E$, and let $T \in \mathcal{L}^r(E,F)$ and $x \in E$. Then there exists a sequence $(x_{n})$ in the span of $\{e_{i} \,\, | \,\, i \in I\}$ that converges uniformly to $x$. We can express each $x_{n}$ as  $x_{n}=\sum_{i \in \alpha_{n}} \kappa_i(n) e_i$, where $\kappa_i(n) \in \mathbb{R},\alpha_{n} \in \mathrm{Fin}(I)$ for each $n \in \mathbb{N}$ and $i \in \alpha_{n}$. Since $T^+$  preserves uniform convergence, we get $T^+(x_{n}) \xrightarrow{{u}} T^+(x)$. We will show next that $T^+(x)\in F$, and thus $T^+ \in \mathcal{L}^r(E,F)$.  By Lemma~\ref{lemmanew}(i),  we obtain $$T^+(x_{n})=\sum_{i \in \alpha_{n}} \kappa_i(n) T^+(e_i)=\sum_{i \in \alpha_{n}} \kappa_i(n) T(e_i)^+ \in F.$$ As $(T^+(x_{n}))_{n \in \mathbb{N}}$ is   uniformly Cauchy in $F^{\delta}$, there exists $v \in F^{\delta}_+$ such that for all $\epsilon>0$ there exists $n_0$ such that $|T^+(x_{n})-T^+(x_{m})|\leq \epsilon v$ for all $n,m \geq n_0$. Now, since $F$ is majorizing in $F^{\delta}$ we may pick $w \in F$ such that $w \geq v$, then  $|T^+(x_{n})-T^+(x_{m})|\leq \epsilon v \leq \epsilon w$ for all $n,m \geq n_0$. Consequently, $(T^+(x_{n}))_{n \in \mathbb{N}}$ is uniformly Cauchy in $F$. By the uniform completeness of $F$, we can find $y \in F$ such that $T^+(x_{n}) \xrightarrow{{u}} y$ in $F$ and thus in $F^{\delta}$ as well.  Finally by the uniqueness of the uniform limit we get $T^+(x)=y \in F$.
\end{proof}

\section{Solutions to some problems}

In the final part of our paper, we apply our main results to present some solutions to the problems mentioned in the introduction of the paper. Recall that an operator $T \in \mathcal{L}^b(E,F)$ is said to be $\sigma$-order continuous if $x_n \xrightarrow{{o}} 0$ implies that $T(x_n)\xrightarrow{{o}} 0$ for all sequences $(x_n)$ in $E$. We denote by $c$ the space of convergent real sequences and $\ell_0^\infty$ the space of eventually constant real sequences. We denote by $\mathbf{e}_n$ the sequence with all zero terms except for the $n$'th which is $1$. Finally, we denote by $\mathbf{1}$ the constant sequence with all terms equal to $1$.  The following result extends and improves \cite[Theorem~4.7, Theorem~5.3]{W:24}.

\begin{theorem}\label{maincor3}
Let $E$ and $F$ be Archimedean Riesz spaces, where $E$ has a countable complete disjoint system of atoms, and let $T \in \mathcal{L}^b(E,F)$. Then $T$ is order continuous if and only if $T$ is $\sigma$-order continuous. Furthermore, if $E=c$, or $E=\ell^\infty_0$, then $T$ is order continuous if and only if  $\sum_{i=1}^n T(\mathbf{e}_i) \xrightarrow{{o}} T(\mathbf{1}).$
\end{theorem}

\begin{proof}
Let $T \in \mathcal{L}^b(E,F)$, and let $F^{\delta}$ be an order completion of $F$ and $i:F \rightarrow F^{\delta}$ the corresponding bipositive map. Then put $\widetilde{T}=iT:E \rightarrow i(F)$, and note that $T$ preserves order convergence of nets if and only if $\widetilde{T}$ preserves order convergence of nets. Thus we may assume without loss of generality that $F=i(F)$. In particular, $T$ can be regarded as an operator taking values in $F^{\delta}$. Let $P$ be the band projection associated with $\mathcal{L}^{oc}(E,F^{\delta})$. We aim to show $T=P(T)$, implying that $T \in \mathcal{L}^{oc}(E,F^{\delta})$. The order continuity of $T$, will then follow by \cite[Corollary~2.9]{GTX:17}. 

Let $\{e_n \mid n \in \mathbb{N}\}$ be a complete disjoint system of atoms in $E$. If $E=c$ or $E=\ell^\infty_0$ we set $\{e_n \mid n \in \mathbb{N}\}$ to be the standard  complete disjoint system of atoms $\{\mathbf{e}_n \mid n \in \mathbb{N}\}$ in $c$ and $\ell^{\infty}_0$. Fix  $x \in E_+$ and define $T_{\alpha}(x) = \sum_{i \in \alpha} \lambda_{e_i}(x) \otimes T(e_i)$ for all $\alpha \in \mathrm{Fin}(\mathbb{N})$, moreover we put $T_n(x)=\sum_{i=1}^n \lambda_{e_i}(x)T(e_i)$ for all $n \in \mathbb{N}$.  Define the function $\phi : \mathbb{N} \to \mathrm{Fin}(\mathbb{N})$ by $\phi(n) = \{1, \dots, n\}$. Note that $n \leq m \Rightarrow \phi(n) \subseteq \phi(m)$, and for all $\alpha \in \mathrm{Fin}(\mathbb{N})$, there exists $n \in \mathbb{N}$ such that $\alpha \subseteq \phi(n)$. Hence, the sequence $(T_n(x))_{n \in \mathbb{N}}$ is a subnet of $(T_\alpha (x))_{\alpha \in \mathrm{Fin}(I)}$. By Lemma~\ref{lemmanew}(ii), $T_{\alpha}(x)\xrightarrow{{o}} P(T)(x)$ in $F^{\delta}$. Therefore, \begin{equation}\label{eq06}T_{n}(x)\xrightarrow{{o}} P(T)(x) \,\,  \text{in} \,\,  F^{\delta}.\end{equation} 

Suppose that $T$ is $\sigma$-order continuous. Since $x=\bigvee_{i=1}^\infty \lambda_{e_i}(x)e_i$ we have $\sum_{i=1}^n \lambda_{e_i}(x)e_i \uparrow x$. The $\sigma$-order continuity of $T$ yields that $T_n(x)=T(\sum_{i=1}^n \lambda_{e_i}(x)e_i) \xrightarrow{{o}} T(x)$ in $F$, and thus in $F^{\delta}$ by \cite[Corollary~2.9]{GTX:17}. Therefore by~(\ref{eq06}) we get $P(T)(x)=T(x)$, and $T=P(T)$ becomes clear.

Now suppose $E=c $  or  $E=\ell^\infty_0$, and $\sum_{i=1}^n T(\mathbf{e}_i) \xrightarrow{{o}} T(\mathbf{1})$, then $\sum_{i=1}^n T(\mathbf{e}_i) \xrightarrow{{o}} T(\mathbf{1})$ in $F^{\delta}$ by \cite[Corollary~2.9]{GTX:17}. In both cases, we have $E=\overline{[\{\mathbf{e}_n\}]}^{u}\oplus[\mathbf{1}]$. Since $T,P(T)$ are order-bounded operators, they  preserve uniform convergence. Hence it is suffices to show that $T$ agrees with $P(T)$ on the set $\{\mathbf{e}_n,\mathbf{1}\,\, | \,\, n \in \mathbb{N}\}$. By Lemma~\ref{lemmanew}(ii) it is immediate that $P(T)(\mathbf{e}_n)=T(\mathbf{e}_n)$ for all $n \in \mathbb{N}$. Finally, observe that $T_n(\mathbf{1})=\sum_{i=1}^n T(\mathbf{e}_i)$, so by ~(\ref{eq06}) we get $P(T)(\mathbf{1})=T(\mathbf{1})$, therefore $T=P(T)$.
\end{proof}

In  \cite[Theorem 3.3]{AW:91} the authors proved that if $E$ is uniformly complete, then  (a) $\mathcal{L}^r(E,\ell_0^\infty)=\mathcal{L}^b(E,\ell_0^\infty)$, (b) $\mathcal{L}^r(E,\ell_0^\infty)$ is a Riesz space, (c) $\mathcal{L}^r(E,\ell_0^\infty)$ has the Riesz-Kantorovich property. The authors then discussed whether the uniformly completeness assumption on $E$ can be removed. We show below that in contrast with statements (a) and (b), statement (c) holds without any assumptions on $E$. In particular by applying  Proposition~\ref{mainprop1} and Theorem~\ref{mainthm} we get the following result.

\begin{theorem}\label{maincor1}
Let $E$ be an Archimedean Riesz space. Then $\mathcal{L}^{r}(E,\ell_0^\infty)$ has the Riesz-Kantorovich property.
\end{theorem}

Let $F$ be an almost Dedekind $\sigma$-complete Riesz space. In \cite[Corollary~4.13, Theorem~4.8]{W:24}, Wickstead proved the following results: (i) $\mathcal{L}^{oc}(\ell_0^{\infty},F)$ is a band of $\mathcal{L}^r(\ell_0^{\infty},F)$, and (ii) $\mathcal{L}^{oc}(\ell_0^{\infty},F)$ is directed. We show below that the technical condition on $F$ can be omitted in the case of statement (i), but cannot be dropped for statement (ii).

\begin{theorem}\label{maincor2}
Let $F$ be an Archimedean Riesz space. Then $\mathcal{L}^{oc}(\ell_0^{\infty},F)$ is a band of $\mathcal{L}^r(\ell_0^{\infty},F)$.
\end{theorem}

\begin{proof}
 By \cite[Theorem~4.1]{AW:91} we have $\mathcal{L}^b(\ell_0^{\infty},F)=\mathcal{L}^r(\ell_0^{\infty},F)$ and by \cite[Theorem~2.1]{AS:05} we have that every order-continuous operator is order-bounded. Thus 
by Proposition ~\ref{mainprop2} and Theorem~\ref{mainthm} it follows that  $\mathcal{L}^{oc}(\ell_0^{\infty},F)\cap \mathcal{L}^r(\ell_0^{\infty},F)=\mathcal{L}^{oc}(\ell_0^{\infty},F)\cap \mathcal{L}^b(\ell_0^{\infty},F)=\mathcal{L}^{oc}(\ell_0^{\infty},F)$  is a band of $\mathcal{L}^r(\ell_0^{\infty},F)$.
\end{proof}

\begin{example}\label{exmp2}{\emph{$\mathcal{L}^{oc}(\ell_{0}^{\infty},F)$ may not be directed.}} Let $\Gamma$ be an uncountable set equipped with the discrete topology, and let $K=\Gamma \cup \{\infty\}$ be the one-point compactification of $\Gamma$. We put $F=C(K)$, the Banach lattice of continuous real-valued functions on $K$. Fix  a distinct sequence $(\gamma_n)$ in $\Gamma$, and consider the indicator function $\mathbbm{1}_{\gamma_n}$ on the singleton $\{\gamma_n\}$ for each $n \in \mathbb{N}$. Then $\mathbbm{1}_{\gamma_n} \in C(K)$ and we have $\mathbbm{1}_{\gamma_n}  \xrightarrow{{o}} 0$ but $\mathbbm{1}_{\gamma_n}  \not\xrightarrow{{o_1}} 0$. This result is credited to Fremlin (see, for example, \cite[p. 141]{SS:74}  and \cite[Example 1]{ACW:20} for a proof). We define the following operator $T: \ell^{\infty}_0 \rightarrow C(K)$,

\[
T(\mathbf{e}_1) = \mathbbm{1}_{\gamma_1}, \quad 
T(\mathbf{e}_n) = \mathbbm{1}_{\gamma_n} - \mathbbm{1}_{\gamma_{n-1}} \text{ for } n \geq 2, \quad \text{and} \quad 
T(\mathbf{1}) = 0.
\]

We observe that $|T(\mathbf{e}_1)|=\mathbbm{1}_{\gamma_1}$ and $|T(\mathbf{e}_n)|= \mathbbm{1}_{\gamma_n}+\mathbbm{1}_{\gamma_{n-1}}$ for all $n \geq 2$. Consequently, for all $n \in \mathbb{N}$ we have: $\sum_{k=1}^n |T(\mathbf{e}_k)|=2\sum_{k=1}^{n-1} \mathbbm{1}_{\gamma_i}+\mathbbm{1}_{\gamma_n} \leq 2 \mathbbm{1} \in C(K)$, where $ \mathbbm{1}$ denotes the constant one function on $K$. By applying \cite[Proposition~2.3]{W:24} we get $T \in  \mathcal{L}^b(\ell^{\infty}_0,C(K))$, we also note that $\sum_{k=1}^n T(\mathbf{e}_k)=\mathbbm{1}_{\gamma_n} \xrightarrow{{o}} 0=T(\mathbf{1})$, thus by applying Theorem~\ref{maincor3} we get  $T \in \mathcal{L}^{oc}(\ell^{\infty}_0,C(K))$. Now, suppose that $\mathcal{L}^{oc}(\ell_{0}^{\infty},C(K))$ is directed, then there exists $S \in \mathcal{L}^{oc}(\ell^{\infty}_0,C(K))$ such that $0,- T \leq S$. Then for all $n \in \mathbb{N}$ we have the following inequality \begin{equation}\label{eqc(k)}\mathbbm{1}_{\gamma_n}=-T(\textbf{1}-\sum_{k=1}^n \mathbf{e}_k) \leq S(\textbf{1}-\sum_{k=1}^n \mathbf{e}_k)=S(\mathbf{1})-\sum_{k=1}^n S(\mathbf{e}_k)=y_n.\end{equation}
Since $S$ is a positive order-continuous operator and $\sum_{k=1}^n \mathbf{e}_k \uparrow \mathbf{1}$ we get  $y_n \downarrow 0$. Therefore by ~(\ref{eqc(k)}) we obtain $\mathbbm{1}_{\gamma_n} \xrightarrow{{o1}} 0$, which leads to a contradiction. Hence $\mathcal{L}^{oc}(\ell_{0}^{\infty},C(K))$ is not directed.
\end{example}

In \cite[Theorem~3.6]{AW:91}, the authors present an example of a pair of atomic Archimedean Riesz spaces \(E\) and \(F\), and an operator \(T \in \mathcal{L}^b(E, F)\) such that \(T \notin \mathcal{L}^r(E, F)\). In the following example, we observe that \(T\) is order continuous, implying that \(\mathcal{L}^{oc}(E, F)\) is not a subset of \(\mathcal{L}^r(E, F)\). Specifically, this result shows that \(\mathcal{L}^{oc}(E, F)\) is not a band of \(\mathcal{L}^r(E, F)\), and consequently, a generalization of the Ogasawara theorem in this form is not possible in general. However, as indicated by Proposition~\ref{mainprop1} and Theorem~\ref{mainthm}, in this case we do have that \(\mathcal{L}^{oc}(E, F) \cap \mathcal{L}^r(E, F)\) is a band of \(\mathcal{L}^r(E, F)\).

\begin{example}\label{exmp1}{\emph{$\mathcal{L}^{oc}(E,F)$ may not be a  subset of $\mathcal{L}^{r}(E,F)$.}} Let $E_K$ denote the space of all double sequence $(x_{n,m})_{n,m \in \mathbb{N}}$ such that 
 
\begin{itemize}
\item[(i)] There exists $n_0 \in \mathbb{N}$ such that $x_{n,m}=x_{\widetilde{n},\widetilde{m}}$ for all $n,\widetilde{n} \geq n_0$ and $m,\widetilde{m} \in \mathbb{N}$.
\item[(ii)] For all $n \in \mathbb{N}$ we have that $(x_{n,m})_{m \in \mathbb{N}} \in \ell^{\infty}_0$. 
\end{itemize}

We denote by $ \ell^\infty_0(\mathbb{N}\times \mathbb{N})$ the space of double sequences which are constant except on a finite set. Let $T:E_K \rightarrow \ell^\infty_0(\mathbb{N}\times \mathbb{N})$ given by the formula

$$(Tx)_{n,m}=x_{n,2m-1}-x_{n,2m}$$

In \cite[Theorem~3.6]{AW:91} it is proved that $T$ is an order bounded operator  that is not regular. We will show next that $T$ is order continuous. Since the domain space $E_K$ is atomic, we can apply the techniques used in the proof of Theorem~\ref{maincor3} to verify the order continuity of $T$. However, because the range space $\ell^\infty_0(\mathbb{N}\times \mathbb{N})$ is also atomic, we can  provide a more straightforward proof of the order continuity of $T$. Let $(x^{\alpha})_{\alpha \in A}$ be a net in $E_K$ such that $x^{\alpha} \xrightarrow{{o}} 0$ in $E_K$. We aim to show that $T(x^{\alpha}) \xrightarrow{{o}} 0$ in $\ell^\infty_0(\mathbb{N}\times \mathbb{N})$.  By passing to a tail of the net, we can assume that $(x^{\alpha})_{\alpha \in A}$ is order bounded. Indeed, by the definition of order convergence the exists $\alpha_0 \in A$ such that $(x^{\alpha})_{\alpha \geq a_0}$ is order bounded (and order convergent to zero). If we show that the net  $(T(x^{\alpha}))_{\alpha \geq a_0}$ is order convergent to zero then it follows that $T(x^\alpha) \xrightarrow{{o}} 0$. Next, observe that both $E_K$ and $\ell^\infty_0(\mathbb{N}\times \mathbb{N})$ are pervasive and thus regular Riesz subspaces of $\mathbb{R}^{\mathbb{N} \times \mathbb{N}}$, the space of all double sequences. Recall also that $o_1$-convergence is equivalent to order convergence for order bounded nets in the order complete space $\mathbb{R}^{\mathbb{N} \times \mathbb{N}}$. Therefore, by \cite[Corollary~2.12]{GTX:17} and \cite[Lemma~8.17]{AB:06} we have $x^{\alpha}_{n,m} \rightarrow 0$ for all $n,m \in \mathbb{N}$. Then we clearly have $T(x^{\alpha})_{n,m}=x^{\alpha}_{n,2m-1}-x^{\alpha}_{n,2m}\rightarrow 0$ for all $n,m \in \mathbb{N}$. Since $T$ is order-bounded, the net $(T(x^{\alpha}))_{\alpha \in A}$ is order-bounded in $\ell^\infty_0(\mathbb{N}\times \mathbb{N})$ and thus by applying again \cite[Corollary~2.12]{GTX:17} and \cite[Lemma~8.17]{AB:06} we have $T(x^{\alpha}) \xrightarrow{{o}} 0$ in $\ell^\infty_0(\mathbb{N}\times \mathbb{N})$, thereby establishing the order continuity of $T$.

\end{example}
\textbf{Acknowledgements.} The authors would like to thank the participants of the Functional Analysis Seminar at the University of Alberta for their comments on our work during the second-named author's visit to the University. We would also like to express our gratitude to V.~G.~Troitsky for providing valuable feedback on an earlier draft of the paper and for the discussions that led to Example~\ref{exmp1}. Finally, we acknowledge the support of the Natural Sciences and Engineering Research Council of Canada (NSERC).


{\footnotesize

\begin{thebibliography}{99}

\bibitem{ACW:20} Abela, K., Chetcuti, E., \& Weber, H. (2020). On different modes of order convergence and some applications. arXiv preprint arXiv:2012.13752.

\bibitem{AS:05} Abramovich, Y., \& Sirotkin, G. (2005). On order convergence of nets. \emph{Positivity}, 9(3), 287--292.

\bibitem{AW:91} Abramovich, Y.~A., \& Wickstead, A.~W. (1991). Regular operators from and into a small Riesz space. \emph{Indagationes Mathematicae}, 2(3), 257--274.

\bibitem{AB:03} Aliprantis, C.~D., \& Burkinshaw, O. (2003). Locally solid Riesz spaces with applications to economics (No. 105). American Mathematical Soc.

\bibitem{AB:06} Aliprantis, C.~D., \& Border, K.~C. (2006). Infinite Dimensional Analysis: A Hitchhiker's Guide. Third edition. Springer.

\bibitem{AL:06} Aliprantis, C.~D., \& Burkinshaw, O. (2006). Positive operators (Vol. 119). Springer Science \& Business Media.

\bibitem{BTW23} O'Brien, M., Troitsky, V.~G., \& van der Walt, J.~H. (2023). Net convergence structures with applications to vector lattices. \emph{Quaestiones Math.}, 46(2).

\bibitem{E:19} Elliot, M. (2019). The Riesz–Kantorovich Formulae. \emph{Positivity}, 23(5), 1245--1259.

\bibitem{GTX:17} Gao, N., Troitsky, V.~G., \& Xanthos, F. (2017). Uo-convergence and its applications to Cesàro means in Banach lattices. \emph{Israel Journal of Mathematics}, 220, 649--689.

\bibitem{HK:21} Hauser, T., \& Kalauch, A. (2021). Order continuity from a topological perspective. \emph{Positivity}, 25(5), 1821--1852.

\bibitem{GK:08} Kalauch, A., \& van Gaans, O. (2008). Bands in pervasive pre-Riesz spaces. \emph{Oper. Matrices}, 2(2), 177--191.

\bibitem{GK:18} Kalauch, A., \& van Gaans, O. (2018). Pre-Riesz Spaces (Vol. 66). Walter de Gruyter GmbH \& Co KG.

\bibitem{SS:74} Schaefer, H.~H., \& Schaefer, H.~H. (1974). Banach lattices and Positive Operators. Springer Berlin Heidelberg.

\bibitem{TT:20} Taylor, M.~A., \& Troitsky, V.~G. (2020). Bibasic sequences in Banach lattices. \emph{Journal of Functional Analysis}, 278(10), 108448.

\bibitem{W:19} Wickstead, A.~W. (2019). When do the regular operators between two Banach lattices form a lattice? Positivity and Noncommutative Analysis: Festschrift in Honour of Ben de Pagter on the Occasion of his 65th Birthday, 591--599.

\bibitem{W:24} Wickstead, A.~W. (2024). Riesz completions of some spaces of regular operators. \emph{Indagationes Mathematicae}, 35(3), 443-458.

\end{thebibliography}




}
\end{document}